\documentclass{scrartcl}
\usepackage{amsmath}
\usepackage{amsfonts}
\usepackage{graphicx}
\newcommand{\bbbn}{{\mathbb N}}
\newcommand{\bbbr}{{\mathbb R}}
\newtheorem{defi}{Definition}[section]
\newtheorem{lemm}[defi]{Lemma}
\newtheorem{rema}[defi]{Remark}
\newenvironment{proof}[1][Proof]{\noindent \emph{#1.} }{\hfill \ %
    \vrule width 0.5em height 0.5em depth 0em\par\medskip}
\title{An exact solver for simple ${\mathcal H}$-matrix systems}
\author{Steffen B\"orm and Jessica G\"ordes}
\date{\today}
\begin{document}
\maketitle
\begin{abstract}
\noindent
Hierarchical matrices (usually abbreviated ${\mathcal H}$-matrices)
are frequently used to construct preconditioners for systems of
linear equations.
Since it is possible to compute approximate inverses or $LU$ factorizations
in ${\mathcal H}$-matrix representation using only ${\mathcal O}(n \log^2 n)$
operations, these preconditioners can be very efficient.

Here we consider an algorithm that allows us to solve a linear
system of equations given in a simple ${\mathcal H}$-matrix format
\emph{exactly} using ${\mathcal O}(n \log^2 n)$ operations.
The central idea of our approach is to avoid computing the inverse
and instead use an efficient representation of the $LU$ factorization
based on low-rank updates performed with the well-known
Sherman-Morrison-Woodbury equation.
\end{abstract}

\section{Introduction}

Hierarchical matrices have been introduced in \cite{HA99,HAKH00} as
a technique for representing certain dense matrices in a data-sparse
and therefore efficient way.
The approach is related to the well-known multipole \cite{RO85,GRRO87}
and panel clustering \cite{HANO88,HANO89} techniques:
instead of approximating a smooth function by a degenerate expansion,
a matrix block is approximated by a low-rank matrix.
The algebraic approach offers the possibility to perform matrix
arithmetic operations efficiently and to treat general matrices.

Already the first papers on ${\mathcal H}$-matrices, e.g., \cite{HA99},
consider the question of solving linear systems of equations with
a system matrix given in ${\mathcal H}$-matrix form.
Until now, the standard approach has been to compute an approximation
of the inverse \cite{GRHA02} or at least an approximate $LU$ factorization
\cite{GR04,BE05}.
Combined with a well-chosen clustering strategy, particularly the $LU$
factorization can be very efficient and rivals algebraic multigrid
algorithms \cite{GRKRLE05}.

Still, even the most refined $LU$ factorization is based on the
${\mathcal H}$-matrix multiplication algorithm, and this algorithm
typically finds only approximations, although these approximations
can be arbitrarily accurate.

In this paper, we present an algorithm that solves a system of linear
equations given in a simple ${\mathcal H}$-matrix representation
\emph{exactly}, at least up to rounding errors introduced by floating
point arithmetic operations.
The algorithm is based on the $LU$ factorization, but while standard
algorithms form the Schur complement explicitly, we handle it implicitly
using the Sherman-Morrison-Woodbury formula.
Due to this approach, the local ranks are preserved, no truncation
to lower rank is required, and therefore the resulting decomposition
can be used to solve the system directly.

The algorithm can be split into two phases:
a setup step computes the quantities describing the factorization of
the matrix, and a solver step then solves the linear system.
The first step requires ${\mathcal O}(n \log^2 n)$ operations, where
$n$ is the matrix dimension, and has to be carried out only once for
a given matrix.
The second step requires only ${\mathcal O}(n \log n)$ operations
and computes the solution for a given right-hand side.

It should be mentioned that there are other algorithms for solving
similar problems:
if the matrix is hierarchically semi-separable, it is possible to
solve systems in ${\mathcal O}(n)$ operations \cite{CHGUPA06},
but this works only if the low-rank blocks are of a very special
nested structure, not for more general ${\mathcal H}$-matrices.

\section{Matrix structure and basic idea}

In order to keep the presentation of the basic ideas simple, we restrict
our attention to the simplest ${\mathcal H}$-matrix structure
\cite{HA99}:

%
%
\begin{defi}[${\mathcal H}$-matrix]
Let $n_0\in\bbbn$.
We let
\begin{equation*}
  {\mathcal H}_0 := \bbbr^{n_0\times n_0}
\end{equation*}
and define ${\mathcal H}$-matrices on higher levels inductively:
let $\ell\in\bbbn$ and $n_\ell := n_0 2^\ell$.
A matrix $A\in\bbbr^{n_\ell\times n_\ell}$ is an element of
${\mathcal H}_\ell\subseteq\bbbr^{n_\ell\times n_\ell}$ if and only if there
are matrices $A_1,A_2\in{\mathcal H}_{\ell-1}$ and vectors
$a_1,a_2,b_1,b_2\in\bbbr^{n_{\ell-1}}$
satisfying
\begin{equation}\label{hmatrix}
  A = \begin{pmatrix}
    A_1 & a_1 b_1^* \\
    a_2 b_2^* & A_2
  \end{pmatrix}.
\end{equation}
We call the set ${\mathcal H}_\ell$ the set of ${\mathcal H}$-matrices
on level $\ell$.
\end{defi}

Given an ${\mathcal H}$-matrix $A\in{\mathcal H}_\ell$ and a
right-hand side vector $z\in\bbbr^{n_\ell}$, we are interested in
finding $x\in\bbbr^{n_\ell}$ with
\begin{equation}\label{problem}
  A x = z.
\end{equation}
In general, this is only possible if $A$ is regular.
Since our algorithm uses a hierarchy of sub-problems to solve the
system, we require $A$ to have a more restrictive property:

%
%
\begin{defi}[Hierarchically regular]
Let $A\in{\mathcal H}_\ell$.
We call $A$ \emph{hierarchically regular} if it is regular and,
in case $\ell>0$, if the submatrices $A_1,A_2\in{\mathcal H}_{\ell-1}$
of its representation (\ref{hmatrix}) are also hierarchically regular.
\end{defi}

We can see that, e.g., positive definite matrices are
hierarchically regular, since all of their diagonal blocks are
positive definite and therefore regular.

Let $\ell\in\bbbn$.
If $A\in{\mathcal H}_\ell$ is hierarchically regular, its block $LU$
decomposition is given by
\begin{equation*}
  A = \begin{pmatrix}
    A_1 & a_1 b_1^*\\
    b_2 a_2^* & A_2
  \end{pmatrix}
  = \begin{pmatrix}
    I & \\
    b_2 a_2^* A_1^{-1} & I
  \end{pmatrix}
  \begin{pmatrix}
    A_1 & a_1 b_1^*\\
    & A_2 - b_2 a_2^* A_1^{-1} a_1 b_1^*
  \end{pmatrix},
\end{equation*}
and we can use this decomposition to solve the linear system
(\ref{problem}).
Note that, since $A$ is hierarchically regular, the matrices
$A$ and $A_1$ are regular, therefore the Schur complement
$A_2 - b_2 a_2^* A_1^{-1} a_1 b_1^*$ also has to be regular.

In order to make handling the Schur complement easier, we introduce
\begin{align}\label{c_def}
  c_A &:= (A_1^{-1})^* a_2, &
  \gamma_A &:= c_A^* a_1 = a_2^* A_1^{-1} a_1
\end{align}
and get
\begin{equation*}
  A = \underbrace{\begin{pmatrix}
    I & \\
    b_2 c_A^* & I
  \end{pmatrix}}_{=:L}
  \underbrace{\begin{pmatrix}
    A_1 & a_1 b_1^*\\
    & A_2 - \gamma_A b_2 b_1^*
  \end{pmatrix}}_{=:U}
  = L U.
\end{equation*}
Now we can consider solving the linear system by block forward
and backward substitution, i.e., we will solve
\begin{align*}
  L y &= z, &
  U x &= y.
\end{align*}
We split the vectors $x$, $y$ and $z$ into subvectors
$x_1,x_2,y_1,y_2,z_1,z_2\in\bbbr^{n_{\ell-1}}$ with
\begin{align}\label{subvectors}
  x &= \begin{pmatrix}
    x_1\\
    x_2
  \end{pmatrix}, &
  y &= \begin{pmatrix}
    y_1\\
    y_2
  \end{pmatrix}, &
  z &= \begin{pmatrix}
    z_1\\
    z_2
  \end{pmatrix},
\end{align}
and can write $Ly=z$ in the form
\begin{align*}
  \begin{pmatrix}
    I & \\
    b_2 c_A^* & I
  \end{pmatrix}
  \begin{pmatrix}
    y_1\\
    y_2
  \end{pmatrix}
  &= \begin{pmatrix}
    z_1\\
    z_2
  \end{pmatrix}, &
  y_1 &= z_1, &
  y_2 &= z_2 - b_2 c_A^* y_1.
\end{align*}
Solving $Ux=y$ for $x$ is a little more involved, since we have
\begin{align*}
  \begin{pmatrix}
    A_1 & a_1 b_1^*\\
    & A_2 - \gamma_A b_2 b_1^*
  \end{pmatrix}
  \begin{pmatrix}
    x_1\\
    x_2
  \end{pmatrix}
  &= \begin{pmatrix}
    y_1\\
    y_2
  \end{pmatrix}, &
  (A_2 - \gamma_A b_2 b_1^*) x_2 &= y_2, &
  A_1 x_1 &= y_1 - a_1 b_1^* x_2
\end{align*}
and have to find a way of solving both sub-problems efficiently.

In order to handle the first equation, we rely on the well-known
Sherman-Morrison-Woodbury equation \cite{SHMO50}.
In our case, it yields
\begin{equation*}
  \left( I + \frac{\gamma_A A_2^{-1} b_2 b_1^*}
                  {1 - \gamma_A b_1^* A_2^{-1} b_2} \right) A_2^{-1}
  = (A_2 - \gamma_A b_2 b_1^*)^{-1}.
\end{equation*}
We simplify the equation by introducing
\begin{align}\label{d_def}
  d_A &:= A_2^{-1} b_2, &
  \delta_A &:= \gamma_A b_1^* d_A = \gamma_A b_1^* A_2^{-1} b_2
\end{align}
and get
\begin{equation}\label{A2_inverse}
  \left( I + \gamma_A \frac{d_A b_1^*}{1-\delta_A} \right) A_2^{-1}
  = (A_2 - \gamma_A b_2 b_1^*)^{-1},
\end{equation}
and $x_2$ can be computed by first recursively finding
$\widehat x_2\in\bbbr^{n_{\ell-1}}$ with
\begin{equation*}
  A_2 \widehat x_2 = y_2
\end{equation*}
and then using the rank one correction
\begin{equation*}
  x_2 = \widehat x_2 + \gamma_A \frac{b_1^* \widehat x_2}
                                     {1 - \delta_A} d_A.
\end{equation*}
Once $x_2$ has been computed, we can proceed to recursively solve
\begin{equation*}
  A_1 x_1 = y_1 - a_1 b_1^* x_2
\end{equation*}
to determine $x_1$, and therefore the solution $x$.

Of course we also need efficient algorithms for computing the
auxiliary vectors $c_A$ and $d_A$ introduced in (\ref{c_def})
and (\ref{d_def}).
Since (\ref{c_def}) involves the inverse of the adjoint of $A_1$,
we require an algorithm for solving systems of the form
\begin{equation}\label{adj_problem}
  A^* x = z.
\end{equation}
Fortunately, we can use the $LU$ factorization to solve this
problem as well:
due to $A = L U$, we also have $A^* = U^* L^*$ and can solve
\begin{align*}
  U^* y &= z, &
  L^* x &= y
\end{align*}
by forward and backward substitution.
Using the subvectors defined in (\ref{subvectors}), the forward
substitution takes the form
\begin{align*}
  \begin{pmatrix}
    A_1^* & \\
    b_1 a_1^* & A_2^* - \bar\gamma_A b_1 b_2^*
  \end{pmatrix}
  \begin{pmatrix}
    y_1\\ y_2
  \end{pmatrix}
  &= \begin{pmatrix}
    z_1\\ z_2
  \end{pmatrix}, &
  A_1^* y_1 &= z_1, &
  (A_2^* - \bar\gamma_A b_1 b_2^*) y_2 &= z_2 - b_1 a_1^* y_1.
\end{align*}
We can compute $y_1$ by recursion and use the adjoint of
equation (\ref{A2_inverse}) to get
\begin{equation*}
  (A_2^*)^{-1}
  \left( I + \bar\gamma_A \frac{b_1 d_A^*}
                  {1-\bar\delta_A} \right)
  = (A_2^* - \bar\gamma_A b_1 b_2^*)^{-1},
\end{equation*}
and this allows us to compute $y_2$ in the form
\begin{align*}
  \widehat z_2 &= z_2 - b_1 a_1^* y_1, &
  \widehat y_2 &= \widehat z_2
     + \bar\gamma_A \frac{d_A^* \widehat z_2}{1-\bar\delta_A} b_1, &
  A_2^* y_2 &= \widehat y_2.
\end{align*}
Now we can turn our attention to the backward substitution to solve
\begin{align*}
  \begin{pmatrix}
    I & c_A b_2^*\\
    & I
  \end{pmatrix}
  \begin{pmatrix}
    x_1\\ x_2
  \end{pmatrix}
  &= \begin{pmatrix}
    y_1\\ y_2
  \end{pmatrix}, &
  x_2 &= y_2, &
  x_1 &= y_1 - c_A b_2^* x_2,
\end{align*}
which fortunately requires only inner products and linear combinations.

\section{Algorithm and complexity}

We have seen that we can compute the solution of the systems (\ref{problem})
and (\ref{adj_problem}) efficiently if we are able to solve sub-problems
involving the two diagonal blocks $A_1$ and $A_2$ and their adjoints.
Assuming that the auxiliary vectors $c_A$ and $d_A$ and the values $\gamma_A$
and $\delta_A$ have already been prepared, this leads to the algorithm
given in Figure~\ref{fi:solve}.

\begin{figure}[ht]
\begin{tabbing}
{\bf procedure} solve($A$, {\bf var} $x$);\\
{\bf begin}\\
\quad\= {\bf if} $\ell=0$ {\bf then}\\
\> \quad\= Solve directly\\
\> {\bf else begin}\\
\> \> $\alpha_1 \gets c_A^* x_1$;\quad
      $x_2 \gets x_2 - \alpha_1 b_2$\\
\> \> solve($A_2$, $x_2$)\\
\> \> $\alpha_2 \gets b_1^* x_2$;\quad
      $\alpha_3 \gets \gamma_A \alpha_2 / (1-\delta_A)$;\quad
      $x_2 \gets x_2 + \alpha_3 d_A$\\
\> \> $\alpha_4 \gets b_1^* x_2$;\quad
      $x_1 \gets x_1 - \alpha_4 a_1$\\
\> \> solve($A_1$, $x_1$)\\
\> {\bf end}\\
{\bf end}
\end{tabbing}
\caption{Solve the linear system: On entry, the vector $x$ contains the
  right-hand side of problem (\ref{problem}). Recursive solves and
  low-rank updates are used to replace it by the solution.
  We assume that the vectors $c_A$ and $d_A$ and the values $\gamma_A$
  and $\delta_A$ have already been prepared.}
\label{fi:solve}
\end{figure}

The algorithm is called with $x=z$ and overwrites the vector $x$
with the solution of system (\ref{problem}).
If $A\in{\mathcal H}_0$, the matrix can be considered small and we
can solve the system directly.
If $A\in{\mathcal H}_\ell$ for $\ell>0$, the recursive procedure
described in the previous section is used:
the first line corresponds to the forward substitution in $L$ and
overwrites $x_2$ by $y_2$.
In the second line, we recursively solve a linear system with the
matrix $A_2$ to overwrite $x_2$ by $\widehat x_2$.
In the third line, we perform the Sherman-Morrison-Woodbury update to get
the ``lower'' half $x_2$ of the solution vector.
In the fourth and fifth line, its ``upper'' half $x_1$ is computed
by first updating the right-hand side and then recursively solving
the remaining system.

The adjoint system (\ref{adj_problem}) can be solved in a similar
fashion by using $A^* = U^* L^*$ as described in the previous
section, this leads to the algorithm given in Figure~\ref{fi:solveadj}.

\begin{figure}[ht]
\begin{tabbing}
{\bf procedure} solveadj($A$, {\bf var} $x$);\\
{\bf begin}\\
\quad\= {\bf if} $\ell=0$ {\bf then}\\
\> \quad\= Solve directly\\
\> {\bf else begin}\\
\> \> solveadj($A_1$, $x_1$)\\
\> \> $\alpha_1 \gets a_1^* x_1$;\quad
      $x_2 \gets x_2 - \alpha_1 b_1$\\
\> \> $\alpha_2 \gets d_A^* x_2$;\quad
      $\alpha_3 \gets \bar\gamma_A \alpha_2 / (1 - \bar\delta_A)$;\quad
      $x_2 \gets x_2 - \alpha_3 b_1$\\
\> \> solveadj($A_2$, $x_2$)\\
\> \> $\alpha_4 \gets b_2^* x_2$;\quad
      $x_1 \gets x_1 - \alpha_4 c_A$\\
\> {\bf end}\\
{\bf end}
\end{tabbing}
\caption{Solve the linear system: On entry, the vector $x$ contains the
  right-hand side of adjoint problem (\ref{adj_problem}). Recursive solves
  and low-rank updates are used to replace it by the solution.
  We assume that the vectors $c_A$ and $d_A$ and the values $\gamma_A$
  and $\delta_A$ have already been prepared.}
\label{fi:solveadj}
\end{figure}

Both algorithms work only if the auxiliary vectors $c_A$ and $d_A$ and
the auxiliary values $\gamma_A$ and $\delta_A$ have already been
prepared.
Fortunately, computing $c_A$ for a matrix $A\in{\mathcal H}_\ell$
requires only solving the adjoint system for $A_1^*\in{\mathcal H}_{\ell-1}$,
and similarly $d_A$ can be computed by solving the system
for $A_2\in{\mathcal H}_{\ell-1}$.
This means that we can prepare these vectors by bootstrapping:
on level $\ell=0$, we do not require the vectors, but we may want
to prepare auxiliary structures for solving efficiently, e.g.,
by computing a suitable factorization of the matrix $A$.
On level $\ell=1$, we have to solve systems on level $\ell-1=0$ in order
to find $c_A$ and $d_A$, but this can be done directly.
Once the vectors on a level $\ell$ have been computed, we can
use them to compute the vectors on level $\ell+1$, until the maximal
level has been reached.
The resulting algorithm is given in Figure~\ref{fi:setup}.

\begin{figure}[ht]
\begin{tabbing}
{\bf procedure} setup($A$);\\
\quad\= {\bf if} $\ell=0$ {\bf then}\\
\> \quad\= Prepare $A$, e.g., compute its factorization\\
\> {\bf else begin}\\
\> \> setup($A_1$);\\
\> \> setup($A_2$);\\
\> \> $c_A \gets a_2$;\quad
      solveadj($A_1$, $c_A$)\\
\> \> $d_A \gets b_2$;\quad
      solve($A_2$, $d_A$)\\
\> \> $\gamma_A \gets c_A^* a_1$;\quad
      $\delta_A \gets \gamma_A b_1^* d_A$\\
\> {\bf end}\\
{\bf end}
\end{tabbing}
\caption{Setup phase: Prepare the vectors $c_A$ and $d_A$ and the
  values $\gamma_A$ and $\delta_A$ according to (\ref{c_def}) and
  (\ref{d_def}) for all submatrices.}
\label{fi:setup}
\end{figure}

Let us now investigate the complexity of the recursive algorithms.
If we denote the storage requirements of the representation
(\ref{hmatrix}) of $A\in{\mathcal H}_\ell$ by $M_\ell$, we find
\begin{align*}
  M_\ell &= \begin{cases}
    n_0^2 &\text{ if } \ell=0,\\
    2 M_{\ell-1} + 4 n_{\ell-1}
    = 2 M_{\ell-1} + 2 n_\ell &\text{ otherwise}
  \end{cases} &
  &\text{ for all } \ell\in\bbbn_0,
\end{align*}
and we can see that this implies
\begin{align*}
  M_\ell &= (2\ell+n_0) n_\ell &
  &\text{ for all } \ell\in\bbbn_0,
\end{align*}
i.e., if we assume $n_0$ to be constant, the storage requirements grow
like ${\mathcal O}(n_\ell \log n_\ell)$.
This is typical for most ${\mathcal H}$-matrix representations.

%
%
\begin{lemm}[Solving]
\label{le:solving}
Assume that there is a constant $C_0\in\bbbr_{>0}$ such that solving the
problems (\ref{problem}) and (\ref{adj_problem}) for level $\ell=0$ requires
not more than $C_0 n_0^2$ operations.

Then for all $\ell\in\bbbn_0$ and $A\in{\mathcal H}_\ell$, the algorithms
given in Figure~\ref{fi:solve} and Figure~\ref{fi:solveadj} require
not more than $(C_0 n_0 + 6\ell) n_\ell$ operations.
\end{lemm}
\begin{proof}
We consider only the algorithm given in Figure~\ref{fi:solve}, since
both algorithms differ only in the sequence the elementary computation
steps are carried out.

We denote the number of operations required on level $\ell\in\bbbn_0$
by $S_\ell\in\bbbn$.

According to our assumption, the algorithm requires not more than
$C_0 n_0^2$ operations on level $\ell=0$, i.e., we have
\begin{equation*}
  S_0 \leq C_0 n_0^2.
\end{equation*}
Let us now consider a level $\ell>0$.
Computing $\alpha_1$, $\alpha_2$ and $\alpha_4$ each requires
$2n_{\ell-1}-1$ operations, while $\alpha_3$ is computed in $3$ operations,
giving us a total of $6n_{\ell-1}$ operations.
The updates of $x_2$ and $x_1$ each require $2 n_{\ell-1}$ operations,
giving us $6 n_{\ell-1}$ operations for all three updates.
Taking the two recursive solves into account, we get
\begin{equation*}
  S_\ell = 2 S_{\ell-1} + 12 n_{\ell-1}
  = 2 S_{\ell-1} + 6 n_\ell.
\end{equation*}
Now we can use a straightforward induction to prove
\begin{align*}
  S_\ell &\leq (C_0 n_0 + 6 \ell) n_\ell &
  &\text{ for all } \ell\in\bbbn_0,
\end{align*}
and this is the desired estimate.
\end{proof}

If we again assume $n_0$ to be constant, we can see that the number of
operations of the solution algorithm grows like
${\mathcal O}(n_\ell \log n_\ell)$, and this can be considered the
optimal complexity given that the storage requirements of the matrix
show the same asymptotic behaviour.

%
%
\begin{lemm}[Preparing]
Assume that there are constants $C_0,\widehat C_0\in\bbbr_{>0}$ such
that solving the problems (\ref{problem}) and (\ref{adj_problem}) for
level $\ell=0$ requires not more than $C_0 n_0^2$ operations and that
preparing, e.g., factoring, the matrix $A$ on this level requires not
more than $\widehat C_0 n_0^3$ operations.

Then for all $\ell\in\bbbn_0$ and $A\in{\mathcal H}_\ell$, the algorithm
given in Figure~\ref{fi:setup} requires not more than
$(\widehat C_0 n_0^2 + (C_0 n_0-1) \ell + 3 \ell^2) n_\ell$ operations.
\end{lemm}
\begin{proof}
We denote the number of operations required on level $\ell\in\bbbn_0$
by $P_\ell\in\bbbn$.

According to our assumption, the algorithm requires not more than
$\widehat C_0 n_0^3$ operations on level $\ell=0$, i.e., we have
\begin{equation*}
  P_0 \leq \widehat C_0 n_0^3.
\end{equation*}
Let us now consider a level $\ell>0$.
Due to Lemma~\ref{le:solving}, computing the vectors $c_A$ and $d_A$
takes not more than $(C_0 n_0 + 6 (\ell-1)) n_{\ell-1}$ per vector.
$\gamma_A$ is computed using $2 n_{\ell-1}-1$ operations, and
$\delta_A$ is computed using $2 n_{\ell-1}$ operations.
This yields
\begin{align*}
  P_\ell &= 2 P_{\ell-1} + 2 (C_0 n_0 + 6(\ell-1)) n_{\ell-1}
                + 4 n_{\ell-1}-1\\
  &< 2 P_{\ell-1} + (C_0 n_0 + 6 \ell - 6) n_\ell
                + 2 n_\ell\\
  &= 2 P_{\ell-1} + 3 (2 \ell - 1) n_\ell + (C_0 n_0 - 1) n_\ell.
\end{align*}
Based on this bound, we can prove
\begin{align*}
  P_\ell &\leq (\widehat C_0 n_0^2 + (C_0 n_0-1) \ell + 3 \ell^2) n_\ell &
  &\text{ for all } \ell\in\bbbn_0
\end{align*}
by a simple induction, and this is the estimate we need.
\end{proof}

Once more assuming that $n_0$ is constant, the number of operations
required to prepare a matrix $A\in{\mathcal H}_\ell$ for the efficient
solver grows like ${\mathcal O}(n_\ell \log^2 n_\ell)$.
The additional logarithmic factor is introduced since each step of
the setup algorithm involves ${\mathcal O}(n_\ell \log n_\ell)$
operations in the solver steps.

%
%
\begin{rema}[Generalization]
The Sherman-Morrison-Woodbury formula can be extended to
matrix updates of rank $k$.
In this case the vectors $a_1,a_2,b_1,b_2$ in (\ref{hmatrix}) can be
replaced by matrices of dimension $n_{\ell-1}\times k$, the coefficients
$\gamma_A$ and $\delta_A$ become $k\times k$ matrices, and instead of
dividing by $1-\delta_A$, we have to solve a $k\times k$ system, but
otherwise the algorithm remains unchanged.

It is not clear if the algorithm can be extended to more general
matrix structures, e.g., those used for three-dimensional integral
equations, since this would mean that it is no longer possible
to treat the Schur complement by a simple low-rank update.
\end{rema}

\section{Numerical experiments}

Since our algorithm computes the exact solution of the problem
(\ref{problem}), we do not have to consider the accuracy of the
computed solution, we only have to investigate the runtime behaviour.
We consider a simple model problem:
$A$ is a symmetric tridiagonal matrix with the value $4$ on the
diagonal and random values between $-1$ and $1$ on the sub- and
superdiagonal.
By the Gershgorin circle theorem this guarantees that $A$ is positive
definite and therefore ${\mathcal H}$-regular, so our algorithm can
be applied.

%
%
\begin{figure}[ht]
\begin{center}
\includegraphics[angle=270,width=12cm]{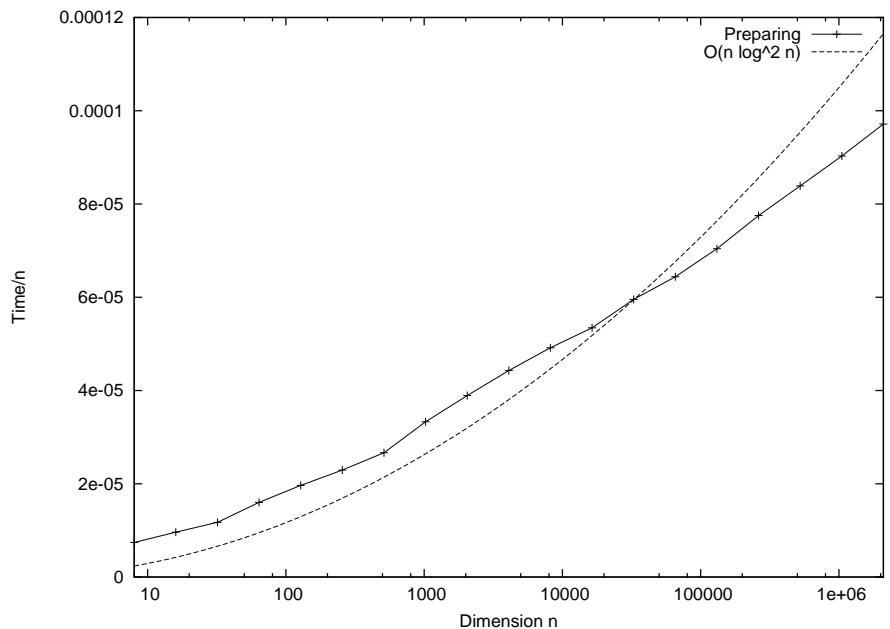}
\end{center}
\caption{Measured time (per degree of freedom) for preparing the
  decomposition}
\label{fi:preparing}
\end{figure}

We use $n_0=2$ and consider matrix dimensions up to $n_0 2^{20}=2097152$.
The runtime for preparing the decomposition is shown in
Figure~\ref{fi:preparing}:
the $x$-axis gives the dimension $n_\ell$ of the matrix in logarithmic
scale, the $y$-axis gives the time per degree of freedom.
We can see that the time grows like ${\mathcal O}(n_\ell \log^2 n_\ell)$,
as predicted by our theory.

%
%
\begin{figure}[ht]
\begin{center}
\includegraphics[angle=270,width=12cm]{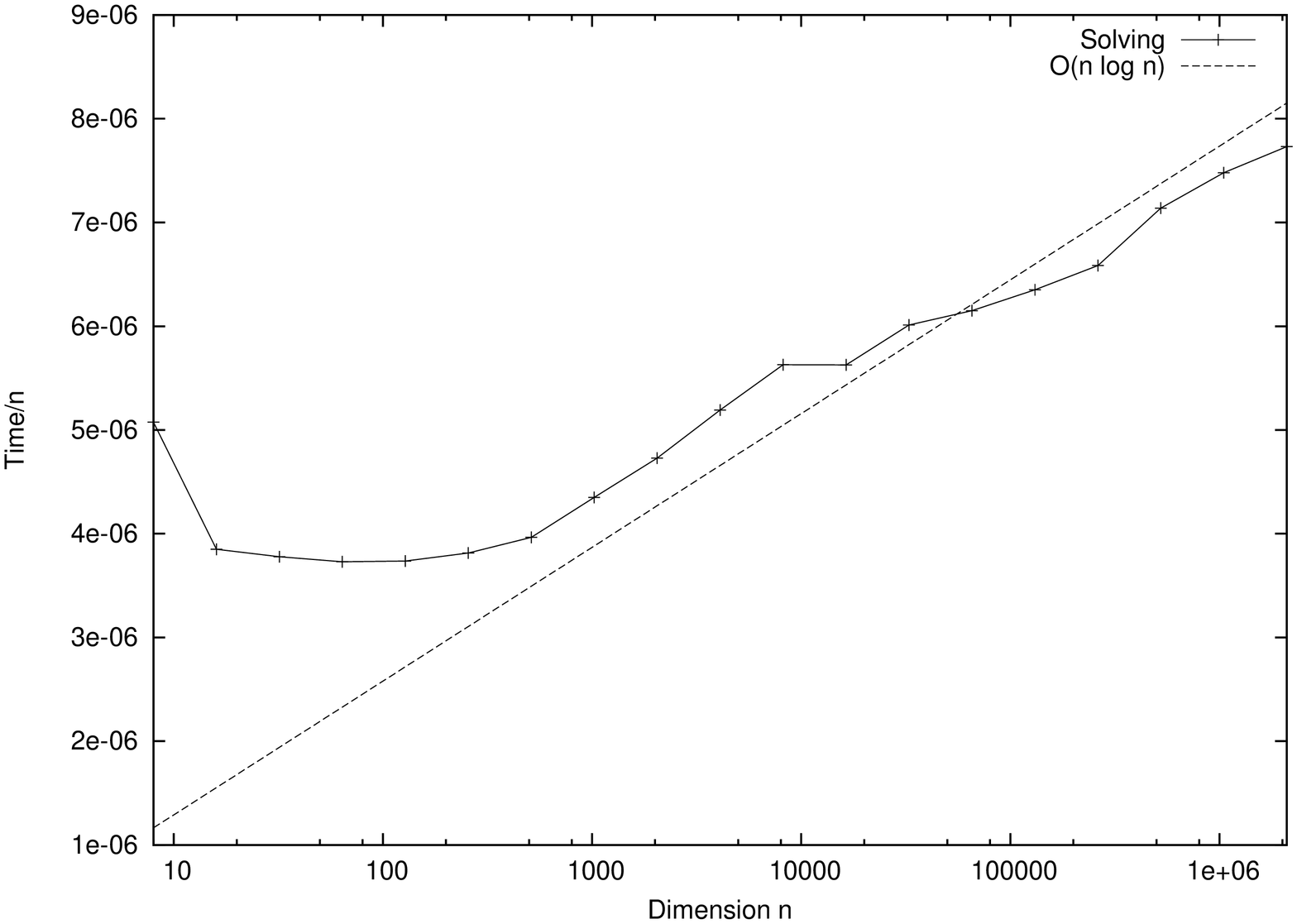}
\end{center}
\caption{Measured time (per degree of freedom) for solving the
  linear system}
\label{fi:solving}
\end{figure}

Figure~\ref{fi:solving} shows the runtime for solving the linear
system once the decomposition has been prepared.
We can see that the time grows like ${\mathcal O}(n_\ell \log n_\ell)$,
agreeing with our theoretical prediction.

\bibliographystyle{plain}
\bibliography{hmatrix}

\begin{thebibliography}{10}

\bibitem{BE05}
M.~Bebendorf.
\newblock Why finite element discretizations can be factored by triangular
  hierarchical matrices.
\newblock {\em SIAM J. of Numer. Anal.}, 45(4):1472--1494, 2007.

\bibitem{CHGUPA06}
S.~Chandrasekaran, M.~Gu, and T.~Pals.
\newblock A fast {ULV} decomposition solver for hierarchically semiseparable
  representations.
\newblock {\em SIAM J. Matrix Anal. Appl.}, 28(3):603--622, 2006.

\bibitem{GR04}
L.~Grasedyck.
\newblock Adaptive recompression of {${\mathcal H}$}-matrices for {BEM}.
\newblock {\em Computing}, 74(3):205--223, 2004.

\bibitem{GRHA02}
L.~Grasedyck and W.~Hackbusch.
\newblock Construction and arithmetics of {${\mathcal{H}}$}-matrices.
\newblock {\em Computing}, 70:295--334, 2003.

\bibitem{GRRO87}
L.~Greengard and V.~Rokhlin.
\newblock A fast algorithm for particle simulations.
\newblock {\em J. Comp. Phys.}, 73:325--348, 1987.

\bibitem{HA99}
W.~Hackbusch.
\newblock A sparse matrix arithmetic based on $\mathcal{H}$-matrices. {P}art
  {I}: {I}ntroduction to $\mathcal{H}$-matrices.
\newblock {\em Computing}, 62:89--108, 1999.

\bibitem{HAKH00}
W.~Hackbusch and B.~N. Khoromskij.
\newblock A sparse matrix arithmetic based on $\mathcal{H}$-matrices. {P}art
  {II}: {A}pplication to multi-dimensional problems.
\newblock {\em Computing}, 64:21--47, 2000.

\bibitem{HANO88}
W.~Hackbusch and Z.~P. Nowak.
\newblock O cloznosti metoda panelej.
\newblock In G.~I. Marchuk, editor, {\em Vycislitel'nye prozessy i sistemy},
  pages 233--244. Nauka, Moskau, 1988.

\bibitem{HANO89}
W.~Hackbusch and Z.~P. Nowak.
\newblock On the fast matrix multiplication in the boundary element method by
  panel clustering.
\newblock {\em Numer. Math.}, 54:463--491, 1989.

\bibitem{GRKRLE05}
S.~LeBorne, L.~Grasedyck, and R.~Kriemann.
\newblock Domain-decomposition {B}ased {H-LU} {P}reconditioners.
\newblock In O.~B. Widlund and D.~E. Keyes, editors, {\em Domain Decomposition
  Methods in Science and Engineering XVI}, volume~55 of {\em Lecture Notes in
  Computational Science and Engineering}, pages 661--668. Springer, 2006.

\bibitem{RO85}
V.~Rokhlin.
\newblock Rapid solution of integral equations of classical potential theory.
\newblock {\em J. Comp. Phys.}, 60:187--207, 1985.

\bibitem{SHMO50}
J.~Sherman and W.~J. Morrison.
\newblock Adjustment of an inverse matrix corresponding to a change in one
  element of a given matrix.
\newblock {\em Annals of Mathematical Statistics}, 21(1):124--127, 1950.

\end{thebibliography}
\end{document}